\documentclass[12pt]{article}
\usepackage{amsmath}
\usepackage{amsfonts}
\usepackage{amsthm}
\usepackage{graphicx}
\usepackage{overpic}
\usepackage{a4wide}
\usepackage{amssymb} 
\usepackage{pictex}
\usepackage{rotating}  
\usepackage{color}
\usepackage{etex} 
\usepackage{enumitem}
\usepackage{accents}

\definecolor{refkey}{rgb}{0,0,1}
\definecolor{labelkey}{rgb}{1,0,0}

\usepackage{comment}

\numberwithin{equation}{section}

\newtheorem{theorem}{Theorem}[section]
\newtheorem{proposition}[theorem]{Proposition}
\newtheorem{lemma}[theorem]{Lemma}
\newtheorem{corollary}[theorem]{Corollary}
\newtheorem{Definition}[theorem]{Definition}
\newenvironment{definition}{\begin{Definition}\rm}{\end{Definition}}
\newtheorem{Remark}[theorem]{Remark}
\newenvironment{remark}{\begin{Remark}\rm}{\end{Remark}}
\newtheorem{Example}[theorem]{Example}
\newenvironment{example}{\begin{Example}\rm}{\end{Example}}
\newtheorem{RHproblem}[theorem]{RH problem}

\usepackage{color}

\newcommand{\C}{\mathbb{C}}

\newcommand{\R}{\mathbb{R}}

\newcommand{\PP}{\mathbb P}

\renewcommand{\hat}{\widehat}
\renewcommand{\tilde}{\widetilde}

\def\det{\mathop{\mathrm{det}}\nolimits}

\usepackage[bookmarksopen, naturalnames]{hyperref}

\begin{document}
\title{A weighted extremal function and equilibrium measure}
\author{Len Bos, Norman Levenberg, Sione Ma`u and Federico Piazzon}
\maketitle 

\begin{abstract}
Let $K=\R^n\subset \C^n$ and 
$Q(x):=\frac{1}{2}\log (1+x^2)$ where $x=(x_1,...,x_n)$ and $x^2 = x_1^2+\cdots +x_n^2$. 
Utilizing extremal functions for convex bodies in $\R^n\subset \C^n$ and Sadullaev's characterization of algebraicity for complex analytic subvarieties of $\C^n$ we prove the following explicit formula for the weighted extremal function $V_{K,Q}$: 
$$V_{K,Q}(z)=\frac{1}{2}\log \bigl( [1+|z|^2] + \{ [1+|z|^2]^2-|1+z^2|^2\}^{1/2})$$
where $z=(z_1,...,z_n)$ and $z^2 = z_1^2+\cdots +z_n^2$. As a corollary, we find that the Alexander capacity $T_{\omega}(\R \PP^n)$ of $\R \PP^n$ is $1/\sqrt 2$. We also compute the Monge-Amp\`ere measure of $V_{K,Q}$: 
$$(dd^cV_{K,Q})^n = n!\frac{1}{(1+x^2)^{\frac{n+1}{2}}}dx.$$
\end{abstract}

\section{Introduction}
For $K\subset \C^n$ compact, define the usual Siciak-Zaharjuta {\it extremal function} 
\begin{equation}\label{vk}
  V_K(z)
  := \max \left\{ 0 ,
    \sup _p \left\{ \frac{1}{deg(p)} \log|p(z)|: p \ \hbox{poly.}, \ ||p||_K:=\max_{z\in K} |p(z)| \leq 1\right\}
    \right\}, \end{equation}
where the supremum is taken over (non-constant) holomorphic polynomials $p$, and let $V_K^*(z):= \limsup_{\zeta \to z} V_K(\zeta)$ be its uppersemicontinuous (usc) regularization. If $K\subset \C^n$ is closed, a nonnegative uppersemicontinuous function $w:K\to [0, \infty)$ with $\{z\in K: w(z)=0\}$ pluripolar is called a weight function on $K$ and $Q(z):=-\log w(z)$ is the {\it potential} of $w$. The associated {\it weighted extremal function} is 
$$V_{K,Q}(z):=\sup \{\frac{1}{deg(p)}\log |p(z)|: p \ \hbox{poly.}, \ ||pe^{-deg(p)Q}||_K\leq 1\}.$$
Note $V_K=V_{K,0}$. For unbounded $K$, the potential $Q$ is required to grow at least like $\log |z|$. If, e.g, 
$$\liminf_{z\in K, \ |z|\to +\infty}\bigl(Q(z)-\log |z|\bigr) > -\infty $$
(we call $Q$ {\it weakly admissible}), then the Monge-Amp\`ere measure $(dd^cV_{K,Q}^*)^n$ may or may not 
have compact support. A priori these extremal functions may be defined in terms of upper envelopes of {\it Lelong class} functions: we write $L(\C^n)$ for the set of all plurisubharmonic (psh) functions $u$ on $\C^n$ with the property that $u(z) - \log |z| = 0(1), \ |z| \to \infty$ and
$$ L^+(\C^n):=\{u\in L(\C^n): u(z)\geq \log^+|z| + C\}$$
where $C$ is a constant depending on $u$. For $K$ compact, either $V_K^*\in L^+(\C^n)$ or $V_K^*\equiv \infty$, this latter case occurring when $K$ is pluripolar; i.e., there exists $u\not \equiv -\infty$ psh on a neighborhood of $K$ with $K\subset \{u=-\infty\}$. In the setting of weakly admissible $Q$ it is a result of \cite{bs} that, provided the function 
$$\sup \{u(z): u\in L(\C^n), \ u\leq Q \ \hbox{on} \ K\}$$
is continuous, it coincides with $V_{K,Q}(z)$. 

If we let $X=\PP^n$ with the usual K\"ahler form $\omega$ normalized so that $\int_{\PP^n} \omega^n =1$, we can define the class of {\it $\omega-$psh functions} (cf., \cite{GZ})
$$PSH(X,\omega) :=\{\phi \in L^1(X): \phi \ \hbox{usc}, \ dd^c\phi +\omega \geq 0\}.$$
Let ${\bf z}:=[z_0:z_1:\cdots :z_n]$ be homogeneous coordinates on $X=\PP^n$. Identifying $\C^n$ with the affine subset of $\PP^n$ given by $\{[1:z_1:\cdots:z_n]\}$, we can identify the $\omega-$psh functions with the Lelong class $L(\C^n)$, i.e., 
$$PSH(X,\omega) \approx L(\C^n),$$
and the bounded (from below) $\omega-$psh functions coincide with the subclass $L^+(\C^n)$: if $\phi \in PSH(X,\omega)$, then
$$u(z)=u(z_1,...,z_n):= \phi ([1:z_1:\cdots:z_n])+\frac{1}{2}\log (1+|z|^2)\in  L(\C^n);$$
if $u\in  L(\C^n)$, define $\phi \in PSH(X,\omega)$ via
$$\phi ([1:z_1:\cdots:z_n])=u(z)-\frac{1}{2}\log (1+|z|^2) \ \hbox{and}$$ 
$$\phi ([0:z_1:\cdots:z_n])=\limsup_{|t|\to \infty, \ t\in \C}[u(tz)-\frac{1}{2}\log (1+|tz|^2)].$$ 
Abusing notation, we write $u= \phi +u_0$ where $u_0(z):=\frac{1}{2}\log (1+|z|^2)$. Given a closed subset $K\subset \PP^n$ and a function $q$ on $K$, we can define a {\it weighted $\omega-$psh extremal function}
$$v_{K,q}({\bf z}):=\sup \{ \phi({\bf z}): \phi \in PSH(X,\omega), \ \phi \leq q \ \hbox{on} \ K\}.$$
Thus if $K\subset \C^n \subset \PP^n$, for $[1:z_1:\cdots:z_n]=[1:z]\in \C^n$ we have
\begin{equation}\label{wtdrel} v_{K,q}([1:z])=\sup \{u(z): u\in L(\C^n), \ u\leq u_0 +q \ \hbox{on} \ K\} -u_0(z)=V_{K,u_0+q}(z)-u_0(z).\end{equation}
If $q=0$, the {\it Alexander capacity} $T_{\omega}(K)$ of $K\subset \PP^n$ was defined in \cite{GZ} as
$$T_{\omega}(K):=\exp {[-\sup_{\PP^n} v_{K,0}]}.$$
This notion has applications in complex dynamics; cf., \cite{DS}.

These extremal psh and $\omega-$psh functions $V_K, V_{K,Q}$ and $v_{K,0}, v_{K,q}$, as well as the homogeneous extremal psh function $H_E$ of $E\subset \C^n$ (whose definition we recall in the next section), are very difficult to compute explicitly. Even when an explicit formula exists, computation of the associated Monge-Amp\`ere measure is problematic. Our main goal in this paper is to utilize a novel approach to explicitly compute $V_{K,Q}$ and $(dd^cV_{K,Q})^n$ for the closed set $K=\R^n\subset \C^n$ and the weight $w(z)=|f(z)|=|\frac{1}{(1+z^2)^{1/2}}|$ where $z^2 =z_1^2 +\cdots + z_n^2$ (see (\ref{magic}) or Theorem \ref{magic1}, and (\ref{monge})). Note the potential $Q(z)$ in this case is the standard K\"ahler potential $u_0(z)$ restricted to $\R^n$. As an application we can calculate the Alexander capacity $T_{\omega}(\R \PP^n)$ of $\R \PP^n$ (Corollary \ref{magiccor}). 

We offer several methods to explicitly compute $V_{K,Q}$. For the first one, we relate this weighted extremal function to:
\begin{enumerate}
\item the extremal function $V_{B_{n+1}}$ of the {\it real $(n+1)-$ball} 
$$B_{n+1}=\{(u_0,...,u_n)\in \R^{n+1}: \sum_{j=0}^n u_j^2\leq1\}$$
in $\R^{n+1}\subset \C^{n+1}$ as well as 
\item the extremal function $V_{\tilde K}$ of the {\it real $n-$sphere}
$$\tilde K =  \{ (u_0,...,u_n)\in \R^{n+1}: \sum_{j=0}^n u_j^2=1\}$$
in $\R^{n+1}$ {\it considered as a compact subset of the complexified $n-$sphere}
$$A :=  \{ (W_0,...,W_n)\in \C^{n+1}: \sum_{j=0}^n W_j^2=1\}$$ 
in $\C^{n+1}$. This function is the {\it Grauert tube function} of $\tilde K$ in $A$; cf., \cite{Z}.
\end{enumerate}

A similar (perhaps simpler) idea is  a relation between $V_{K,Q}$ and 
\begin{enumerate}
\item the extremal function $V_{B_{n}}$ of the {\it real $n-$ball} 
 $$B_n:= \{(u_1,...,u_n)\in \R^{n}: \sum_{j=1}^n u_j^2\leq1\}$$
 in $\R^n\subset \C^{n}$ and
\item the homogeneous extremal function $H_S$ of the {\it real $n-$upper hemisphere}
$$S:=\{(u_0,...,u_n)\in \R^{n+1}: \sum_{j=0}^n u_j^2\leq1, \ u_0 >0\}$$
in $\R^{n+1}$ considered as a subset of $A$
\end{enumerate}
\noindent obtained by projecting $S$ onto $B_n$. 

In both cases we appeal to two well-known  and highly non-trivial results: 
\begin{enumerate}
\item using Theorem \ref{blmthm} (or \cite{bar}) we have a foliation of $\C^n\setminus B_n$ (and $\C^{n+1} \setminus B_{n+1}$) by complex ellipses on which $V_{B_n}$ ($V_{B_{n+1}})$ is harmonic; and 
\item using Theorem \ref{sadthm} we have $V_{\tilde K}$ (and $H_S$) is locally bounded on $A$ and is maximal on $A\setminus \tilde K$ (on $A\setminus S$).
\end{enumerate}
\noindent See the next section for statements of Theorems \ref{blmthm} and \ref{sadthm} and section 4 for details of these relations.

Bloom (cf., \cite{BL} and \cite{Bloomtams}) introduced a technique to switch back and forth between certain pluripotential-theoretic notions in $\C^{n+1}$ and their weighted counterparts in $\C^n$; we recall this in the next section. In section 3, we discuss a modification of Bloom's technique suitable for special weights $w$ and we use this modification in section 4 to construct a formula for $V_{K,Q}$ on a neighborhood of $\R^n$ for the set $K=\R^n\subset \C^n$ and weight $w(z)=|f(z)|=|\frac{1}{(1+z^2)^{1/2}}|$. This formula gives an explicit candidate $u\in L(\C^n)$ for $V_{K,Q}$. In section 5 we give another ``geometric'' interpretation of $u$ by observing a relationship with the Lie ball
$$L_n :=\{z=(z_1,...,z_n)\in \C^n: |z|^2 +\{ |z|^4 - |z^2|^2\}^{1/2} \leq 1\}$$
which we use to explicitly compute that $(dd^cu)^n=0$ on $\C^n\setminus \R^n$, verifying that $u=V_{K,Q}$. As a corollary, we compute the Alexander capacity $T_{\omega}(\R \PP^n)$ of $\R \PP^n$. Finally, section 6 utilizes results from \cite{blmr} to compute an explicit formula for the Monge-Amp\`ere measure $(dd^cV_{K,Q})^n$.

\section{Known results on extremal functions} In this section, we list some results and connections about extremal functions, all of which will be utilized.

One particular situation where we know much information about $V_K$ is when $K$ is a convex body in $\R^n$; i.e., $K\subset \R^n$ is compact, convex and int$_{\R^n}K\not =\emptyset$.
\begin{theorem}\label{blmthm}
Let $K\subset \R^n$ be a convex body. Through every point $z\in\C^n\setminus K$ there is either a complex ellipse $E$ with $z\in E$ such that $V_K$ restricted to $E$ is harmonic 
on $E\setminus K$, or there is a complexified real line $L$ with $z\in L$ such that $V_K$ is harmonic on $L\setminus K$. For such $E$, $E\cap K$ is a real ellipse  inscribed in $K$ with the property that for its given eccentricity and orientation, it is the ellipse with largest area  
completely contained in $K$; for such $L$, $L\cap K$ is the longest line segment (for its given direction) completely contained in $K$.
\end{theorem}

We refer the reader to Theorem 5.2 and Section 6 of \cite{blm2}; see also \cite{blm}. The ellipses and lines in Theorem \ref{blmthm} have parametrizations of the form
$$
F(\zeta) = a + c\zeta + \frac{\bar c}{\zeta},
$$
$a\in\R^n$, $c\in\C^n$, $\zeta \in \C$ with $V_K(F(\zeta))=\log^+|\zeta|$ ($\bar c$ denotes the component-wise complex conjugate of $c$). These are higher dimensional analogs of the classical Joukowski function $\zeta\mapsto \frac{1}{2}(\zeta + \frac{1}{\zeta})$. For $K = B_n$, the real unit ball in $\R^n
\subset \C^n$, the real ellipses $E\cap B_n$ and lines $L\cap B_n$ in Theorem \ref{blmthm} are symmetric with respect to the origin and, other than great circles in the real boundary of $B_n$, each $E\cap B_n$ and $L\cap B_n$ hits this real boundary at exactly two antipodal points. Lundin proved \cite{lunpre}, \cite{bar} that 
\begin{equation}
\label{eq:realball} V_K(z) =\frac{1}{2} \log h(|z|^2 + |z^2 -
1|),
\end{equation}
where $|z|^2 = \sum |z_j|^2, \ z^2 = \sum z_j^2,$ and $h$ is the inverse Joukowski map 
$h(\frac{1}{2}(t + \frac{1}{t})) = t$ for $1 \leq t \in \R$. In
this example, the Monge-Amp\`ere measure $(dd^cV_K)^n$ has the
explicit form
$$(dd^cV_K)^n = n! \ vol(K) \ \frac{dx}{(1-|x|^2)^{\frac{1}{2}}}:=n! \ vol(K) \ \frac{dx_1 \wedge \cdots \wedge dx_n}{(1- |x|^2)^{\frac{1}{2}}} $$
(see also (\ref{eq:monge})).

We may consider the class
$$H:=\{u \in L(\C^n): \ u(tz) =\log {|t|} +u(z), \ t\in \C, \ z \in \C^n \} $$ 
of {\it logarithmically homogeneous} psh functions and, for $E\subset \C^n$, the {\it homogeneous extremal function of $E$}
denoted by $H_E^*$ where
$$H_E(z):=\max [0,\sup \{u(z):u \in H, \ u\leq 0 \ \hbox{on} \
    E\}]. $$
Note that $H_E(z)\leq V_E(z)$. If $E$ is compact, we have
$$H_E(z)=\max [0,\sup \{\frac{1}{deg (h)}\log {|h(z)|}: h \ \hbox{homogeneous
polynomial}, \ ||h||_E\leq 1\}]. $$
The $H-$principle of Siciak (cf., \cite{Kl}) gives a one-to-one correspondence between
\begin{enumerate}
\item homogeneous polynomials $H_d(t,z)$ of a fixed degree $d$ in $\C_t\times \C^n_z$ and polynomials $p_d(z)=H_d(1,z)$ of degree $d$ in $\C^n_z$ via 
$$H_d(t,z):=t^dp_d(z/t); $$
\item psh functions $h(t,z)$ in $H(\C_t\times \C^n_z)$ and psh functions $u(z)=h(1,z)$ in $L(\C^n_z)$ via 
$$h(t,z)=\log {|t|} +u(z/t) \ \hbox{if} \ t\not = 0; \ h(0,z):=\limsup_{(t,\zeta)\to (0,z)}h(t,z);  $$
\item  extremal functions $V_E$ of $E\subset \C^n_z$ and homogeneous extremal functions $H_{1\times E}$ via 2.; i.e.,
\begin{equation}\label{easyfcn}V_E(z)=H_{1\times E}(1,z). \end{equation}
\end{enumerate}

\noindent To expand upon 3., given a compact set $E\subset \C^n$, if one forms the circled set ($S$ is circled means $z\in S \iff e^{i\theta}z\in S$)
$$Z(E):=\{(t,tz)\in \C^{n+1}: z\in E, \ |t|=1\} \subset \C^{n+1},$$
then 
$$H_{Z(E)}(1,z) = V_E(z);$$
indeed, for $t\not = 0$,
 $$H_{Z(E)}(t,z) = V_E(z/t)+\log |t|.$$

Note that $Z(E)$ is the ``circling'' of the set $\{1\}\times E\subset \C^{n+1}$. In general, if $E\subset \C^n$, the set 
$$E_c:=\{e^{i\theta}z: z\in E, \ \theta \in \R\}$$
is the smallest circled set containing $E$. If $E$ is compact, then $\hat E_c$, the polynomial hull of $E_c$, is given by 
$$\hat E_c=\{tz: \ z\in E, \ |t|\leq 1\}$$
which coincides with the {\it homogeneous polynomial hull} of $E$:
$$\hat E_{hom}:=\{z\in \C^n: |p(z)|\leq ||p||_E \ \hbox{for all homogeneous polynomials} \ p\}.$$
We have $H_{E_c}=V_{E_c}$. For future use we remark that if $E\subset F$ with $H_E=H_F=V_F$, it is {\it not} necessarily true that $V_E=H_E$. As a simple example, we can take $E=B_n$, the real unit ball, and $F=\hat E_c=\hat E_{hom}$. Then $F=L_n$, the {\it Lie ball} 
$$L_n =\{z=(z_1,...,z_n)\in \C^n: |z|^2 +\{ |z|^4 - |z^2|^2\}^{1/2} \leq 1\}$$
(see section 5). Here, $V_{B_n}\not = V_{L_n}$.

More generally, if $K\subset \C^n$ is closed and $w$ is a weight function on $K$, we can form the circled set 
$$Z(K,Q):= \{(t,tz)\in \C^{n+1}: z\in E, \ |t|=w(z)\}$$
and then
$$H_{Z(K,Q)} (1,z) = V_{K,Q}(z);$$
indeed, for $t\not = 0$,
 $$H_{Z(K,Q)} (t,z) = V_{K,Q}(z/t)+\log |t|.$$
This is the device utilized by Bloom (cf., \cite{BL} and \cite{Bloomtams}) alluded to in the introduction. 

Finally, we mention the following beautiful result of Sadullaev \cite{Sad}. 
\begin{theorem}\label{sadthm} Let $A$ be a pure $m-$dimensional, irreducible analytic subvariety of $\C^n$ where $1\leq m \leq n-1$. Then $A$ is algebraic if and only if for some (all) $K\subset A$ compact and nonpluripolar in $A$, $V_K$ in (\ref{vk}) is locally bounded on $A$.

\end{theorem}

\noindent Note that $A$ and hence $K$ is pluripolar in $\C^n$ so $V_K^*\equiv \infty$; moreover, $V_K=\infty$ on $\C^n\setminus A$. In this setting, $V_K|_A$ (precisely, its usc regularization in $A$) is maximal on the regular points $A^{reg}$ of $A$ outside of $K$; i.e., $(dd^cV_K|_A)^m=0$ there, and $V_K|_A \in L(A)$. Here $L(A)$ is the set of psh functions $u$ on $A$ ($u$ is psh on $A^{reg}$ and locally bounded above on $A$) with the property that $u(z) - \log |z| = 0(1)$ as $|z| \to \infty$ through points in $A$, see \cite{Sad}.

\section{Relating extremal functions} Let $K\subset \C^n$ be closed and let $f$ be holomorphic on a neighborhood $\Omega$ of $K$. We define $F:\Omega \subset \C^n\to \C^{n+1}$ as 
$$F(z):=(f(z),zf(z))=W=(W_0,W')=(W_0,W_1,...,W_n)$$
where $W'=(W_1,...,W_n)$. Thus 
$$W_0= f(z), \ W_1 = z_1f(z), ..., \ W_n=z_nf(z).$$
Moreover we assume there exists a polynomial $P=P(z_0,z)$ in $\C^{n+1}$ with $P(f(z),z)=0$ for $z\in \Omega$; i.e., $f$ is {\it algebraic}. Taking such a polynomial $P$ of minimal degree, let 
\begin{equation}\label{variety} A:=\{W\in \C^{n+1}:P(W_0,W'/W_0)=P(W_0,W_1/W_0,...,W_n/W_0)=0 \}.\end{equation}
Note that writing $P(W_0,W'/W_0)=\tilde P(W_0,W')/W_0^s$ where $\tilde P$ is a polynomial in $\C^{n+1}$ and $s$ is the degree of $P(z_0,z)$ in $z$ we see that $A$ differs from the algebraic variety 
$$\tilde A:=\{W\in \C^{n+1}:\tilde P(W_0,W')=0\}$$
by at most the set of points in $A$ where $W_0=0$, which is pluripolar in $A$. 
Thus we can apply Sadullaev's Theorem \ref{sadthm} to nonpluripolar subsets of $A$. Now $P(f(z),z)=0$ for $z\in \Omega$ says that 
$$F(\Omega)=\{(f(z),zf(z)): z \in \Omega\}\subset A.$$
We can define a weight function $w(z):=|f(z)|$ which is well defined on all of $\Omega$ and in particular on $K$; as usual, we set
\begin{equation}\label{need2}Q(z):=-\log w(z) = -\log |f(z)|.\end{equation}
We will need our potentials defined in (\ref{need2}) to satisfy 
\begin{equation}\label{need}Q(z):=\max \{-\log |W_0|: W\in A, \ W'/W_0=z\}\end{equation}
and we mention that (\ref{need}) can give an a priori definition of a potential for those $z\in \C^n$ at which there exist $W\in A$ with $W'/W_0=z$. 

We observe that for $K\subset \Omega$, we have two natural associated subsets of $A$: 
\begin{enumerate}
\item $\tilde K:= \{W\in A: W'/W_0\in K\}$ and 
\item $F(K)=\{W=F(z)\in A: z \in K\}$.
\end{enumerate}
\noindent Note that $F(K)\subset \tilde K$ and the inclusion can be strict.

\begin{proposition}\label{exfcn} Let $K\subset \C^n$ be closed with $Q$ in (\ref{need2}) satisfying (\ref{need}). If $F(K)$ is nonpluripolar in $A$, 
$$V_{K,Q}(z)-Q(z)\leq H_{F(K)}(W) \ \hbox{for} \ z\in \Omega \ \hbox{with} \ f(z)\not =0 $$
where the inequality is valid for $W=F(z)\in F(\Omega)$.
\end{proposition}

\noindent This reduces to (\ref{easyfcn}) if $w(z)\equiv 1$ ($Q(z)\equiv 0$) in which case $F(K)= \{1\}\times K$.

\begin{remark} In general, Proposition \ref{exfcn} only gives estimates for $V_{K,Q}(z)$ if $z\in \Omega$ and $f(z)\not =0$. We will use this and Lemma \ref{lowerest} in the next section to get a formula for $V_{K,Q}(z)$ when $K=\R^n\subset \C^n$ and the weight $w(z)=|f(z)|=|\frac{1}{(1+z^2)^{1/2}}|$ for $z$ in a neighborhood $\Omega$ of $\R^n$ and in section 5 we will verify that this formula is valid on all of $\C^n$. 

\end{remark}
\begin{proof}
First note that for $z\in K$ and $W=F(z)\in F(K)$, given a polynomial $p$ in $\C^n$,
$$|w(z)^{deg p}p(z)|=|f(z)|^{deg p}|p(z)|= |W_0^{deg p} p(W'/W_0)|=|\tilde p(W)|$$ where $\tilde p$ is the homogenization of $p$. Thus $||w^{deg p}p||_K\leq 1$ implies $|\tilde p|\leq 1$ on $F(K)$. Now fix $z\in \Omega$ at which $f(z)\not =0$ (so $Q(z)<\infty$) and fix $\epsilon >0$. Choose a polynomial $p=p(z)$ with $||w^{deg p}p||_K\leq 1$ and 
$$\frac{1}{deg p} \log |p(z)|\geq V_{K,Q}(z) -\epsilon.$$
Thus
$$V_{K,Q}(z) -\epsilon -Q(z)\leq \frac{1}{deg p} \log |p(z)|-Q(z).$$
For $W\in A$ with $W_0\not =0$ and $W'/W_0=z$, the above inequality reads:
$$V_{K,Q}(z) -\epsilon -Q(z)\leq \frac{1}{deg p} \log |p(W'/W_0)|-Q(W'/W_0)\leq \frac{1}{deg p} \log |p(W'/W_0)|+\log |W_0|$$
from (\ref{need}). But
$$\frac{1}{deg p} \log |p(W'/W_0)|+\log |W_0|=\frac{1}{deg p} \log |W_0^{deg p}p(W'/W_0)|=\frac{1}{deg \tilde p} \log |\tilde p(W)|.$$
This shows that 
$$V_{K,Q}(z) -\epsilon -Q(z)\leq \sup \{\frac{1}{deg \tilde p} \log |\tilde p(W)|: |\tilde p|\leq 1 \ \hbox{on}  \ F(K)\}\leq H_{F(K)}(W).$$

\end{proof}

Next we prove a lower bound involving $\tilde K$ which will be applicable in our special case.

\begin{definition} \label{three4} \rm 
Let $A\subset\C^{n+1}$ be an algebraic hypersurface.  We say that $A$ is \emph{bounded on lines  through the origin}  if there exists a uniform constant $c\geq 1$ such that for all $W\in A$, 
 if $\alpha W\in A$ also holds for some $\alpha\in\C$,  then $|\alpha|\leq c$.
\end{definition}

\begin{example} \label{three5} \rm
A simple example of a hypersurface bounded on lines through the origin is one given by an equation of the form $p(W)=1$, where $p$ is a homogeneous polynomial.  In this case, if $\alpha W\in A$ then $$1=p(\alpha W)=\alpha^{deg p}p(W)=\alpha^{deg p},$$ 
so $\alpha$ must be a root of unity. Hence we may take $c=1$. 
\end{example}



In order to get a lower bound on $V_{K,Q}-Q$ we need to be able to extend $Q$ to a function in $L(\C^n)$.

\begin{lemma}\label{lowerest} 
Let $K\subset\C^n$ and let $Q(z)=-\log |f(z)|$ with $f$ defined and holomorphic on $\Omega\supset K$. Define $A$ as in (\ref{variety}) and assume $Q$ satisfies (\ref{need}). We suppose $A$ is bounded on lines through the origin, $\tilde K$ is a nonpluripolar subset of $A$, and that $Q$ has an extension to $\C^n$ (which we still call $Q$) satisfying (\ref{need}) such that $Q\in L(\C^n)$. Then given $z\in\C^n$, 
$$
H_{\tilde K}(W)\leq V_{\tilde K}(W)\leq V_{K,Q}(z)-Q(z)
$$
for all $W=(W_0,W')\in A$ with $W'/W_0=z$.
\end{lemma}

\begin{proof} The left-hand inequality $H_{\tilde K}(W)\leq V_{\tilde K}(W)$ is immediate. For the right-hand inequality, we 
first note that $V_{\tilde K}(W)\in L(A)$ if $\tilde K$ is nonpluripolar in $A$.  Hence there exists a constant $C\in\R$ such that
$$
V_{\tilde K}(W) \leq \log|W| + C
= \log|W_0| + \frac{1}{2}\log(1+|W'/W_0|^2) + C
$$
 for all $W\in A$ with $W_0\not = 0$. 

Define the function
$$
U(z):= \max\{V_{\tilde K}(W): W\in A, W'/W_0=z\} + Q(z).
$$
Note that the right-hand side is a locally finite maximum since $A$ is an algebraic hypersurface.  Away from the singular points $A^{sing}$ of $A$ one can write $V_{\tilde K}(W)$ as a psh function in $z$ by composing it with a local inverse of the map $A\ni W\mapsto z=W'/W_0\in\C^n$. Hence $U$ is psh off the pluripolar set $$\{z\in\C^n: z=W'/W_0 \ \hbox{ for some }  \ W\in A^{sing}\},$$ and hence psh everywhere since it is clearly locally bounded above on $\C^n$.  

Also, since $V_{\tilde K}=0$ on $\tilde K$ it follows that $U\leq Q$ on $K$.  We now verify that $U\in L(\C^n)$ by checking its growth.  By the definitions of $U$ and $Q$ and (\ref{need}), given $z\in\C^n$ there exist $W,V\in A$, with 
$z=W'/W_0= V'/V_0$, such that 
$$
U(z) = V_{\tilde K}(W) + Q(z) \ \hbox{ and } \ Q(z)=-\log|V_0|. 
$$
Note that $W=\alpha V$, and since $A$ is uniformly bounded on lines through the origin, there is a uniform constant $c$ (independent of $W,V$) such that $|\alpha|\leq c$.   We then compute 
\begin{eqnarray*}
 U(z) = V_{\tilde K}(W) - \log|V_0| 
&\leq& V_{\tilde K}(W)  - \log|W_0| + \log c \\ 
&\leq& \log|W| + C - \log|W_0| + \log c \\
&=& \log|W/W_0| + C +\log c =  \tfrac{1}{2}\log(1+|z|^2) + C +\log c
\end{eqnarray*}
where $C>0$ exists since $V_{\tilde K}\in L(A)$.   Hence $U\in L(\C^n)$, and since $U\leq Q$ on $K$ this means that 
$U(z)\leq V_{K,Q}(z)$.  By the definition of $U$, 
$$
V_{\tilde K}(W)+Q(z)\leq V_{K,Q}(z)
$$
for all $W\in A$ such that $W'/W_0=z$, which completes the proof.
\end{proof}

The situation of Lemma \ref{lowerest} will be the setting of our example in the next section.


\section{The weight $w(z)=|\frac{1}{(1+z^2)^{1/2}}|$ and $K=\R^n$} We consider the closed set $K=\R^n\subset \C^n$ and the weight $w(z)=|f(z)|=|\frac{1}{(1+z^2)^{1/2}}|$ where $z^2 =z_1^2 +\cdots + z_n^2$. Note that $f(z)\not =0$ and we may extend $Q(z)=-\log |f(z)|$ to all of $\C^n$ as $Q(z)=\frac{1}{2}\log |1+z^2|\in L(\C^n)$. Since
$$(1+z^2)\cdot f(z)^2 -1 =0,$$
we take 
$$P(z_0, z) = (1+z^2)z_0^2 -1.$$
Here,
 $$A = \{W:P(W_0,W'/W_0)=(1+W'^2/W_0^2)W_0^2-1= W_0^2+W'^2-1=0\}$$
 is the complexified sphere in $\C^{n+1}$. From Definition \ref{three4} and Example \ref{three5}, $A$ is bounded on lines  through the origin. Note that $f$ is clearly holomorphic in a neighborhood of $\R^n$; thus we can take, e.g., $\Omega=\{z= x+iy \in \C^n: y^2=y_1^2 +\cdots + y_n^2 < s <1\}$ in Proposition \ref{exfcn} and Lemma \ref{lowerest} where $z_j=x_j+iy_j$. Condition (\ref{need}) holds for $Q(z)=\frac{1}{2}\log |1+z^2|\in L(\C^n)$ at $z\in \C^n$ for which there exist $W\in A$ with $W'/W_0=z$ since $W=(W_0,W')\in A$ implies $W_0=\pm \sqrt {1-(W')^2}$ so that $|W_0|$ is the same for each choice of $W_0$. We have
 $$F(K)= \{(f(z),zf(z)): z=(z_1,...,z_n) \in K=\R^n\}=\{(\frac{1}{(1+x^2)^{1/2}},\frac{x}{(1+x^2)^{1/2}}):x\in \R^n\}.$$
 Writing $u_j = {\rm Re} W_j$, we see that 
 $$F(K)=\{ (u_0,...,u_n)\in \R^{n+1}: \sum_{j=0}^n u_j^2=1, \ u_0 >0\}.$$ 
 On the other hand,
 $$\tilde K = \{W\in A: W'/W_0\in K\}= \{ (u_0,...,u_n)\in \R^{n+1}: \sum_{j=0}^n u_j^2=1\}.$$ 
Clearly $\tilde K$ is nonpluripolar in $A$ which completes the verification that Lemma \ref{lowerest} is applicable. 
 We also observe that since for any homogeneous polynomial $h=h(W_0,...,W_n)$ we have 
 $$|h(-u_0,u_1,...,u_n)|=|h(u_0,-u_1,...,-u_n)|,$$
 the homogeneous polynomial hulls of $\tilde K$ and $\overline {F(K)}$ in $\C^{n+1}$ coincide so that $H_{\tilde K} = H_{\overline{F(K)}}$ in $A$. Since 
 $$\overline{F(K)}\setminus F(K)=\{ (u_0,...,u_n)\in \R^{n+1}: \sum_{j=0}^n u_j^2=1, \ u_0 =0\}\subset A\cap 
 \{ W_0=0\}$$
 is a pluripolar subset of $A$, 
  \begin{equation}\label{hulleq} H_{\tilde K} =  H_{F(K)}\end{equation}
 on $A\setminus P$ where $P\subset A$ is pluripolar in $A$. Combining (\ref{hulleq}) with Proposition \ref{exfcn} and Lemma \ref{lowerest}, we 
have 
 \begin{equation}\label{fullin} H_{\tilde K} (W)=V_{\tilde K}(W) = V_{K,Q}(z) -Q(z)= H_{F(K)}(W)\end{equation}
 for $z\in \tilde \Omega :=\Omega \setminus \tilde P$ and $W=F(z)$ where $\tilde P$ is pluripolar in $\C^n$.

 To compute the extremal functions in this example, we first consider $V_{\tilde K}$ in $A$. Let 
 $$B:=B_{n+1}=\{(u_0,...,u_n)\in \R^{n+1}: \sum_{j=0}^n u_j^2\leq1\}$$
 be the real $(n+1)-$ball in $\C^{n+1}$. 
 
 \begin{proposition} We have
 $$V_B(W)=V_{\tilde K}(W)$$
 for $W\in A$.
 
 \end{proposition}
 
 \begin{proof} Clearly $V_B|_A \leq V_{\tilde K}$. To show equality holds, the idea is that if we consider the complexified extremal ellipses $L_{\alpha}$ as in Theorem \ref{blmthm} for $B$ whose real points $S_{\alpha}$ are great circles on $\tilde K$, the boundary of $B$ in $\R^{n+1}$, then the union of these varieties fill out $A$: $\cup_{\alpha} L_{\alpha} =A$. Since $V_B|_{L_{\alpha}}$ is {\it harmonic}, we must have $V_B|_{L_{\alpha}} \geq  V_{\tilde K}|_{L_{\alpha}}$ so that $V_B|_A = V_{\tilde K}$. 
 
 To see that $\cup_{\alpha} L_{\alpha} =A$, we first show $A\subset \cup_{\alpha} L_{\alpha}$. If $W\in A\setminus \tilde K$, then $W$ lies on {\it some} complexified extremal ellipse $L$ whose real points $E$ are an inscribed ellipse in $B$ with boundary in $\tilde K$ (and $V_B|_L$ is harmonic). If $L\not = L_{\alpha}$ for some $\alpha$, then $E\cap \tilde K$ consists of two antipodal points $\pm p$. By rotating coordinates we may assume $\pm p = (\pm 1,0,...,0)$ and 
 $$E\subset \{(u_0,...,u_n):  \ u_2=\cdots =u_n=0\}.$$
We have two cases: 
 \begin{enumerate}
 \item $E=\{(u_0,...,u_n): |u_0|\leq 1, \ u_1=0, \ u_2=\cdots =u_n=0\}$, a real interval:
 
 \noindent In this case
 $$L=\{(W_0,0,...,0): W_0\in \C \}.$$
 But then $L\cap A =\{(W_0,0,...,0): W_0=\pm 1 \}=\{\pm p\} \subset \tilde K$, contradicting $W\in A\setminus \tilde K$.

 \item $E=\{(u_0,...,u_n): u_0^2+u_1^2/r^2=1, \ u_2=\cdots =u_n=0\}$ where $0<r<1$, a nondegenerate ellipse:
 
 \noindent In this case, 
 $$L:=\{(W_0,...,W_n): W_0^2+W_1^2/r^2=1, \ W_2=\cdots =W_n=0\}.$$
 But then if $W\in L\cap A$ we have
 $$ W_0^2+W_1^2/r^2=1= W_0^2+W_1^2$$
 so that $W_1=\cdots =W_n=0$ and $W_0^2=1$; i.e., $L\cap A =\{\pm p\} \subset \tilde K$ which again contradicts $W\in A\setminus \tilde K$.

 \end{enumerate}
 
 For the reverse inclusion, recall that the variety $A$ is defined by $\sum_{j=0}^nW_j^2=1$. If $W=u+iv$ with $u,v\in \R^{n+1}$, we have
  $$\sum_{j=0}^nW_j^2 = \sum_{j=0}^n[u_j^2-v_j^2]+ 2i\sum_{j=0}^nu_jv_j.$$ 
  Thus for $W=u+iv \in A$, we have 
  $$\sum_{j=0}^nu_jv_j=0.$$
   If we take an orthogonal transformation $T$ on $\R^{n+1}$, then, by definition, $T$ preserves Euclidean lengths in $\R^{n+1}$; i.e., $\sum_{j=0}^n u_j^2=1=\sum_{j=0}^n (T(u)_j)^2=1$. Moreover, if $u,v$ are orthogonal; i.e., $\sum_{j=0}^nu_jv_j=0$, then $\sum_{j=0}^n(T(u))_j\cdot (T(v))_j =0$. Extending $T$ to a complex-linear map on $\C^{n+1}$ via 
  $$T(W)=T(u+iv):= T(u) +iT(v),$$
  we see that if $W\in A$, then $\sum_{j=0}^n(T(u))_j\cdot (T(v))_j =0$ so that
  $$\sum_{j=0}^n(T(W)_j)^2 = \sum_{j=0}^n[(T(u)_j)^2-(T(v)_j)^2]=\sum_{j=0}^n[u_j^2-v_j^2]=1.$$
  Thus $T$ preserves $A$.
  
  Clearly the ellipse 
  $$L_0:=\{(W_0,...,W_n): W_0^2+W_1^2=1, \ W_2=\cdots =W_n=0\}$$
 corresponding to the great circle 
 $S_0:=\{(u_0,...,u_n): u_0^2+u_1^2=1, \ u_2=\cdots =u_n=0\}$ lies in $A$ and any other great circle $S_{\alpha}$ can be mapped to $S_0$ via an orthogonal transformation $T_{\alpha}$. From the previous paragraph, we conclude that $ \cup_{\alpha} L_{\alpha}\subset A$.

 \end{proof}
 
 \noindent We use the Lundin formula for $V_B$ in (\ref{eq:realball}):
 $$V_B(W)=\frac{1}{2}\log h\bigl( |W|^2+|W^2-1|\bigr)$$ 
 where $h(t)=t+\sqrt{t^2-1}$ for $t\in \C\setminus [-1,1]$. Now the formula for $V_{\tilde K}$ can only be valid on $A$; and indeed, since $W^2=1$ on $A$, by the previous proposition we obtain
 $$V_{\tilde K}(W)=\frac{1}{2}\log h (|W|^2), \ W\in A.$$
Note that since the real sphere $\tilde K$ and the complexified sphere $A$ are invariant under real rotations, the Monge-Amp\`ere measure 
$$(dd^cV_{\tilde K} (W))^n=(dd^c \frac{1}{2}\log h( |W|^2))^n$$ 
must be invariant under real rotations as well and hence is normalized surface area measure on the real sphere $\tilde K$. This can also be seen as a consequence of $V_{\tilde K}$ being the {\it Grauert tube function} for $\tilde K$ in $A$ as $(dd^cV_{\tilde K} (W))^n$ gives the volume form $dV_g$ on $\tilde K$ corresponding to the standard Riemannian metric $g$ there (cf., \cite{Z}).

 Getting back to the calculation of $V_{K,Q}$, note that since $W=(\frac{1}{(1+z^2)^{1/2}},\frac{z}{(1+z^2)^{1/2}})$,
 $$|W|^2:=|W_0|^2+|W_1|^2+\cdots +|W_n|^2=\frac{1+|z|^2}{|1+z^2|}.$$
 Plugging in to (\ref{fullin})
 $$V_{\tilde K}(W)= V_B(W) =V_{K,Q}(z)-Q(z)=V_{K,Q}(z)- \frac{1}{2}\log |1+z^2|$$
 gives 
\begin{equation}\label{magic}V_{K,Q}(z)=\frac{1}{2}\log \bigl( [1+|z|^2] + \{ [1+|z|^2]^2-|1+z^2|^2\}^{1/2}\bigr)\end{equation}
for $z\in \tilde \Omega$. We show in section 5 that this formula does indeed give us the extremal function $V_{K,Q}(z)$ for all $z\in \C^n$. 
 
 A similar observation leads to another derivation of the above formula. Consider $\overline {F(K)}$ as the upper hemisphere 
 $$S:=\{(u_0,...,u_n)\in \R^{n+1}: \sum_{j=0}^n u_j^2 =1, \ u_0 \geq 0\}$$
 in $\R^{n+1}$ and let $\pi: \R^{n+1}\to \R^n$ be the projection $\pi(u_0,...,u_n)=(u_1,...,u_n)$ which we extend to $\pi: \C^{n+1}\to \C^n$ via $\pi(W_0,...,W_n)=(W_1,...,W_n)$. Then 
 $$\pi(S) = B_n:= \{(u_1,...,u_n)\in \R^{n}: \sum_{j=1}^n u_j^2\leq1\}$$
 is the real $n-$ball in $\C^{n}$. Each great semicircle $C_{\alpha}$ in $S$ -- these are simply half of the $L_{\alpha}$'s from before -- projects to half of an inscribed ellipse $E_{\alpha}$ in $B_n$, while the other half of $E_{\alpha}$ is the projection of the great semicircle given by the negative $u_1,...,u_n$ coordinates of $C_{\alpha}$ (still in $F(K)$, i.e., with $u_0>0$).  As before, the complexification $E^*_{\alpha}$ of the ellipses $E_{\alpha}$ correspond to complexifications of the great circles.

 \begin{proposition} \label{semi} We have
 $$H_{F(K)}(W_0,...,W_n)=V_{B_n}(\pi(W))=V_{B_n}(W_1,...,W_n)=V_{B_n}(W')\leq V_{\tilde K}(W_0,...,W_n)$$
 for $W=(W_0,...,W_n)=(W_0,W')\in A$.
 \end{proposition}
 
 \begin{proof} Clearly $V_{B_n}(\pi(W))\leq V_{\tilde K}(W)$. For the inequality $H_{F(K)}(W)\leq V_{B_n}(\pi(W))$, note that for $W\in A$ with $W=(W_0,W')$, we have $\pi^{-1}(W')=(\pm W_0,W')\in A$ but the value of $H_{F(K)}$ is the same at both of these points. Thus $W'\to H_{F(K)}(\pi^{-1}(W'))$ is a well-defined function of $W'$ for $W\in A$ which is clearly in $L(\C^n)$ (in the $W'$ variables) and nonpositive if $W'\in B_n$; hence $H_{F(K)}(\pi^{-1}(W'))\leq V_{B_n}(W')$.
 
 \end{proof}
 
 From (\ref{fullin}), 
 $$ H_{\tilde K} (W)=V_{\tilde K}(W) = V_{K,Q}(z) -Q(z)= H_{F(K)}(W)$$
for $z\in \tilde \Omega$ and $W=F(z)$ so that we have equality for such $W$ in Proposition \ref{semi} and an alternate way of computing $V_{K,Q}$. From the Lundin formula, for $(W_0,W')\in A$ we have $W_0^2+W'^2=1$ so
 $$V_{B_n}(W')=\frac{1}{2}\log h\bigl( |W'|^2+|W'^2-1|\bigr)=\frac{1}{2}\log h( |W|^2).$$
and we get the same formula (\ref{magic})
 \[V_{K,Q}(z)=\frac{1}{2}\log \bigl( [1+|z|^2] + \{ [1+|z|^2]^2-|1+z^2|^2\}^{1/2}\bigr)\]
 for $z\in \tilde \Omega$. 
 
 \begin{remark}\label{onevarem} Note that for $n=1$, it is easy to see that
\begin{equation}\label{onevar} V_{K,Q}(z)=\max [\log |z-i|, \log |z+i|]\end{equation}
which agrees with formula (\ref{magic}).
 
 \end{remark}

\section{Relation with Lie ball and maximality of $V_{K,Q}$} One way of describing the Lie ball $L_n\subset \C^n$ is that it is the  homogeneous polynomiall hull $\hat{(B_n)}_{hom}$ of the real ball 
$$B_n:=\{x=(x_1,...,x_n)\in \R^n: x^2 = x_1^2+\cdots +x_n^2\leq 1\}.$$
A formula for $L_n$ is given by
$$L_n =\{z=(z_1,...,z_n)\in \C^n: |z|^2 +\{ |z|^4 - |z^2|^2\}^{1/2} \leq 1\}.$$
Note that (by definition) $L_n$ is circled. Writing $Z:=(z_0,z)=(z_0,z_1,...,z_n)\in \C^{n+1}$, 
$$L_{n+1}=\{Z\in \C^{n+1}: |Z|^2 +\{ |Z|^4 - |Z^2|^2\}^{1/2} \leq 1\}.$$
The (homogeneous) Siciak-Zaharjuta extremal function of this (circled) set is 
$$H_{L_{n+1}}(Z)=V_{L_{n+1}}(Z)= \frac{1}{2}\log^+ \bigl(|Z|^2 +\{ |Z|^4 - |Z^2|^2\}^{1/2}\bigr).$$
Thus
$$V_{L_{n+1}}(1,z)= \frac{1}{2}\log \bigl([1+|z|^2] +\{ [1+|z|^2]^2 - |1+z^2|^2\}^{1/2}\bigr)$$
so that from (\ref{magic})
$$V_{K,Q}(z)= V_{L_{n+1}}(1,z)$$
for $z\in \tilde \Omega$.

The extremal function $V_{L_{n+1}}(Z)$ for the Lie ball in $\C^{n+1}$ is maximal outside $L_{n+1}$ and, since 
$$V_{L_{n+1}}(\lambda Z)= \log |\lambda| + V_{L_{n+1}}(Z)$$
for $Z\in \partial L_{n+1}$ and $\lambda \in \C$ with $|\lambda|>1$, we see that $V_{L_{n+1}}$ is harmonic 
on complex lines through the origin (in the complement of $L_{n+1}$). Thus for each $Z\not \in L_{n+1}$, the vector $Z$ is an eigenvector of the complex Hessian of $V_{L_{n+1}}$ at $Z$ with eigenvalue $0$. We will use this to show: {\sl for $z\not \in \R^n$, the vector ${\rm Im} z$ is an eigenvector of the complex Hessian of the function $V_{K,Q}(z)$ defined in (\ref{magic}) at $z$ with eigenvalue $0$}.

To this end, let $u:\C^n\to\R$ denote our candidate function for $V_{K,Q}$ where $K=\R^n\subset \C^n$ and the weight $w(z)=|f(z)|=|\frac{1}{(1+z^2)^{1/2}}|$, i.e., for $z\in \C^n$, define
\[u(z):=\frac{1}{2}\log \bigl( [1+|z|^2] + \{ [1+|z|^2]^2-|1+z^2|^2\}^{1/2}\bigr).\]
Let $U:\C^{n+1}\to \R$ denote its homogenization, i.e,
\[U(Z)=\frac{1}{2}\log \bigl(|Z|^2+ \{|Z|^4-|Z^2|^2\}\bigr)\]
with $Z:=(z_0,z)\in \C^{n+1},$ so that $u(z)=U(1,z)$. From above, $\max[0,U(Z)]$ is the extremal function for the Lie ball $L_{n+1}$, and since $U(Z)$ is psh, so is $u(z)$. Also, $U$ is symmetric as a function of its arguments and has the property that $U(\overline{Z})=U(Z)$; in particular it follows that
\[\frac{\partial^2U}{\partial Z_j\partial \overline{Z}_k}(\overline{Z})=
\frac{\partial^2U}{\partial Z_j\partial \overline{Z}_k}(Z).\]

Now, for any function $v$, let $H_v(z)$ denote the complex Hessian of $v$ evaluated at the point $z.$ For any fixed $Z\in\C^{n+1}$ and $\lambda\in\C,$
\[U(\lambda Z)=U(Z)+\log|\lambda|,\]
which is harmonic as a function of $\lambda$ for $\lambda \not = 0$. It follows that
\begin{equation} \label{atz}
H_U(Z)Z=0\in \C^{n+1},\quad \forall Z\in \C^{n+1}\setminus \{0\}
\end{equation}
and that
\begin{equation} \label{atzbar}
H_U(Z)\overline{Z}=H_U(\overline{Z})\overline{Z}= 0\in \C^{n+1},\quad \forall Z\in \C^{n+1}\setminus \{0\}.
\end{equation}

Equivalently, equation \eqref{atz} says that, for $0\le j\le n,$
\[\sum_{k=0}^{n} \frac{\partial^2U}{\partial Z_j\partial \overline{Z}_k}(Z)
\times Z_k=0.\]
But then, for $1\le j\le n,$ we have
\[\sum_{k=1}^{n} \frac{\partial^2U}{\partial Z_j\partial \overline{Z}_k}(Z)
\times Z_k=-\frac{\partial^2U}{\partial Z_j\partial \overline{Z}_{0}}(Z)\times Z_{0}.\]
Evaluating at $Z=(1,z)$ we obtain
\[\sum_{k=1}^{n} \frac{\partial^2U}{\partial Z_j\partial \overline{Z}_k}(1,z)
\times z_k=-\frac{\partial^2U}{\partial Z_j\partial \overline{Z}_{0}}(1,z)\times 1,\]
i.e.,
\[\sum_{k=1}^{n} \frac{\partial^2u}{\partial z_j\partial \overline{z}_k}(z)
\times z_k=-\frac{\partial^2U}{\partial Z_j\partial \overline{Z}_{0}}(1,z).\]
Similarly, from \eqref{atzbar} we obtain, for $1\le j\le n,$
\[\sum_{k=1}^{n} \frac{\partial^2U}{\partial Z_j\partial \overline{Z}_k}(Z)
\times \overline{Z}_k=-\frac{\partial^2U}{\partial Z_j\partial \overline{Z}_{0}}(Z)\times \overline{Z}_{0}\]
so that evaluating at $Z=(1,z)$ gives
\[\sum_{k=1}^{n} \frac{\partial^2u}{\partial z_j\partial \overline{z}_k}(z)
\times \overline{z}_k=-\frac{\partial^2U}{\partial Z_j\partial \overline{Z}_{0}}(1,z).\]
Consequently, 
\[H_u(z)z=H_u(z)\overline{z}, \,\,\hbox{i.e.,}\,\, H_u(z)(z-\overline{z})=0.\]
In particular, for $z\neq \overline{z},$ i.e., $z\notin \R^n,$ ${\rm det}(H_u(z))=0,$ i.e., $(dd^c u)^n=0$ (note as $u$ is psh, $H_u(z)$ is a positive semi-definite matrix).

Since the function $u$ is maximal on $\C^n\setminus \R^n$ and $u(x)= Q(x)=\frac{1}{2}\log (1+x^2)$ for $x\in \R^n$ we have proved the following:

\begin{theorem} \label{magic1} For $K=\R^n\subset \C^n$ and weight $w(z)=|f(z)|=|\frac{1}{(1+z^2)^{1/2}}|$, 
$$V_{K,Q}(z)=\frac{1}{2}\log \bigl( [1+|z|^2] + \{ [1+|z|^2]^2-|1+z^2|^2\}^{1/2}\bigr), \ z\in \C^n.$$

\end{theorem}


Note that from (\ref{wtdrel}), since the K\"ahler potential $u_0(x)=Q(x)$ for $x\in K=\R^n$, 
$$V_{K,Q}(z)=  u_0(z)+ v_{K,0}([1:z]).$$
Thus we have found a formula for the (unweighted) extremal function of $\R \PP^n$, the real points of $\PP^n$.

\begin{corollary} \label{magiccor}The unweighted $\omega-$psh extremal function of $\R \PP^n$ is given by  
$$v_{\R \PP^n,0}([1:z])=\frac{1}{2}\log \bigl( [1+|z|^2] + \{ [1+|z|^2]^2-|1+z^2|^2\}^{1/2}\bigr)-u_0(z)$$
\begin{equation}\label{extrwt} =\frac{1}{2}\log \bigl( 1+[1-\frac{|1+z^2|^2}{(1+|z|^2)^2}]^{1/2}\bigr)\end{equation}
for $[1:z]\in \C^n$ and
\begin{equation}\label{extrwt2}v_{\R \PP^n,0}([0:z])=\frac{1}{2}\log \bigl( 1+[1-\frac{|z^2|^2}{(|z|^2)^2}]^{1/2}\bigr).\end{equation}

\end{corollary}

Since $|1+z^2|\leq 1+|z|^2$ (and $|z^2|\leq |z|^2$), we see that, e.g., upon taking $z=i(1/\sqrt n,...,1/\sqrt n)$ in (\ref{extrwt}) or letting $z\to 0$ in (\ref{extrwt2}), 
$$\sup_{{\bf z}\in \PP^n}v_{\R \PP^n,0}({\bf z})= \frac{1}{2}\log 2.$$
This gives the exact value of the Alexander capacity $T_{\omega}(\R \PP^n)$ of $\R \PP^n$ in Example 5.12 of \cite{GZ}:
$$ T_{\omega}(\R \PP^n)=1/\sqrt 2.$$
We remark that Dinh and Sibony had observed that the value of the Alexander capacity $T_{\omega}(\R \PP^n)$ was independent of $n$ (Proposition A.6 in \cite{DS}). 
\section{Calculation of $(dd^cV_{K,Q})^n$ with $V_{K,Q}$ in (\ref{magic})}

 We will compute $(dd^cV_{K,Q})^n$ for $V_{K,Q}$ in (\ref{magic}) after discussing some differential geometry. Let $\delta(x;y)$ be a Finsler metric where $x\in \R^n$ and $y\in \R^n$ is a tangent vector at $x$. The Busemann density associated to this Finsler metric is 
$$\omega(x)=\frac{vol(\hbox{Euclidean unit ball in $\R^n$})}{vol(B_x)}$$
where
$$B_x:= \{y: \delta(x;y)\leq 1\}.$$
The Holmes-Thompson density associated to $\delta(x;y)$ is
$$\tilde \omega(x)=\frac{vol(B_x^*)}{vol(\hbox{Euclidean unit ball in $\R^n$})}$$
where
$$B_x^*:=\{y: \delta(x;y)\leq 1\}^*=\{x:x\cdot y= x^ty\leq 1 \ \hbox{for all} \ y\in B_x\}$$
is the dual unit ball. Here $x^t$ denotes the transpose of the (vector) matrix $x$. Finsler metrics arise naturally in pluripotential theory in the following setting: if $K=\bar \Omega$ where $\Omega$ is a bounded domain in $\R^n\subset \C^n$, the quantity 
\begin{equation}\label{baran}\delta_B(x;y):=\limsup_{t\to 0^+}\frac{V_K(x+ity)}{t}=\limsup_{t\to 0^+}\frac{V_K(x+ity)-V_K(x)}{t} \end{equation}
for $x\in K$ and $y\in \R^n$ defines a Finsler metric called the {\it Baran pseudometric} (cf., \cite{blw}). It is generally not Riemannian: such a situation yields more information on these densities. 

\begin{proposition} \label{ball} Suppose
$$\delta(x;y)^2=y^tG(x)y$$
is a Riemannian metric; i.e., $G(x)$ is a positive definite matrix. Then 
$$vol(B_x^*)\cdot vol(B_x)=1 \ \hbox{and} \ vol(B_x^*)=\sqrt {\det G(x)}.$$

\end{proposition}

\begin{proof} Writing $G(x)=H^t(x)H(x)$, we have
$$\delta(x;y)^2=y^tG(x)y=y^tH^t(x)H(x)y.$$
Letting $||\cdot||_2$ denote the standard Euclidean ($l^2$) norm, we then have
$$B_x=\{y\in \R^N: ||H(x)y||_2\leq 1\}=H^{-1}(x)\bigl( \hbox{unit ball in $l^2-$norm})$$
and
$$B_x^*=H(x)^t\bigl( \hbox{unit ball in $l^2-$norm}).$$
Hence $vol(B_x^*)\cdot vol(B_x)=1$ and 
$$vol\bigl(\{y:\delta(x;y)\leq 1\}^*\bigr)=vol(B_x^*)=\det H(x)=\sqrt {\det G(x)}.$$
\end{proof}

Motivated by (\ref{baran}) and Theorem \ref{main} below, for $u(z)=V_{K,Q}(z)$ in (\ref{magic}), we will show that the limit  
$$\delta_u(x;y):=\lim_{t\to 0^+}\frac{u(x+ity)-u(x)}{t}$$
exists. Fixing $x\in \R^n$ and $y\in \R^n$, let
$$F(t):=u(x+ity)=\frac{1}{2}\log \{ (1+x^2+t^2y^2) + 2 [t^2y^2+t^2x^2y^2-(x\cdot ty)^2]^{1/2}\}$$
$$=\frac{1}{2}\log \{ (1+x^2+t^2y^2) + 2 t[y^2+x^2y^2-(x\cdot y)^2]^{1/2}\}.$$
It follows that 
$$\delta_u(x;y)=F'(0)=\frac{1}{2}\frac{2 [y^2+x^2y^2-(x\cdot y)^2]^{1/2}}{1+x^2}=\frac{ [y^2+x^2y^2-(x\cdot y)^2]^{1/2}}{1+x^2}.$$
We write
$$\delta_u^2(x;y)=\frac{ y^2+x^2y^2-(x\cdot y)^2}{(1+x^2)^2}=y^tG(x)y$$
where
$$G(x):=\frac{(1+x^2)I -xx^t}{(1+x^2)^2}.$$
Since this matrix is positive definite, $\delta_u(x;y)$ defines a Riemannian metric.

We analyze this further. The eigenvalues of the rank one matrix $xx^t\in \R^{n\times n}$ are $x^2,0,\ldots,0$ for 
$$(xx^t)x = x(x^tx) = x^2\cdot x;$$
and clearly $v\perp x$ implies $(xx^t)v=x(x^tv)=0$. The eigenvalues of $(1+x^2)I -xx^t$ are then 
$$(1+x^2)-x^2, \ (1+x^2)-0, \ldots, \ (1+x^2)-0  \ = \ 1,1+x^2,\ldots, 1+x^2$$
and the eigenvalues of $G(x)$ are
$$\frac{1}{(1+x^2)^2}, \frac{1}{1+x^2},\ldots,  \frac{1}{1+x^2}.$$
This shows $G(x)$ is, indeed, positive definite (it is clearly symmetric) and 
$$\det G(x)=\frac{1}{(1+x^2)^{n+1}}.$$ 
From Proposition \ref{ball}, 
$$vol(B_x^*) = \sqrt {\det G(x)}=\frac{1}{(1+x^2)^{\frac{n+1}{2}}}=\frac{1}{vol(B_x)}.$$
In particular, the Busemann and Holmes-Thompson densities associated to $\delta_u(x;y)$ are
\begin{equation}\label{dense}\frac{1}{(1+x^2)^{\frac{n+1}{2}}}\end{equation}
up to normalization. Note from (\ref{onevar}) in Remark \ref{onevarem} this agrees with the density of $\Delta V_{K,Q}$ with respect to Lebesgue measure $dx$ on $\R$ if $n=1$ and this will be the case for the density of $ (dd^cV_{K,Q})^n$ with respect to Lebesgue measure $dx$ on $\R^n$ for $n>1$ as well. For motivation, we recall the main result of \cite{blmr} (see \cite{bt} for the symmetric case $K=-K$):

\begin{theorem}
\label{main} Let $K$ be a convex body and $V_K$ its
Siciak-Zaharjuta extremal function. The limit
\begin{equation}
\label{dblim} \delta (x;y):=\lim_{t\to
0^+}{V_K(x+ity)\over t}
\end{equation}
exists for each $x\in {\rm int}_{\R^n}K$ and $y\in \R^n$ and 
\begin{equation}
    \label{eq:monge}
    (dd^cV_K)^n=\lambda(x)dx  \ \hbox{where} \ \lambda(x)=n!vol (\{y: \delta  (x;y)\leq 1\}^*)=n!vol(B_x^*).
    \end{equation}
    \end{theorem}
\noindent The conclusion of Theorem \ref{main} required Proposition 4.4 of \cite{blmr}:

\begin{proposition} \label{propbaran} Let $D\subset \C^n$ and let $\Omega :=D\cap
\R^n$.  Let $v$ be a nonnegative locally bounded psh function on $D$ which
satisfies:

$\begin{array}{rl}
 i. & \Omega=\{v=0\}; \\
ii.  & (dd^cv)^n=0 \; {\mbox{on}} \; D\setminus \Omega; \\
iii. & (dd^cv)^n=\lambda(x)dx \; on \; \Omega; \\
iv. & for \; all \; x\in \Omega, \ y \in \R^n, \; the \; limit
\end{array}$
$$h(x,y):=\lim_{t\to 0^+} {v(x+ity)\over t}  \; exists \; and \; is \; continuous \; on \; \Omega \times i\R^n;$$

$\begin{array}{rl} v. & for \; all \; x\in \Omega, y\to h(x,y) \; is \; a \;
norm.
\end{array}$

Then
$$\lambda(x)=n! {\rm vol} \{y:h(x,y)\leq 1\}^*$$
and $\lambda(x)$ is a continuous function on $\Omega$.
\end{proposition}

\begin{theorem} For $V_{K,Q}$ in (\ref{magic}),
\begin{equation} \label{monge} (dd^cV_{K,Q})^n=n! \frac{1}{(1+x^2)^{\frac{n+1}{2}}}dx.\end{equation}
\end{theorem}

\begin{proof}
Recall we extended $Q(x)=\frac{1}{2}\log (1+x^2)$ on $\R^n$ to all of $\C^n$ as $$Q(z)=\frac{1}{2}\log |1+z^2|\in L(\C^n).$$ With this extension of $Q$, and writing $u:=V_{K,Q}$, we claim 
 \begin{enumerate}
 \item $Q$ is pluriharmonic on $\C^n\setminus V$ where $V=\{z\in \C^n: 1+z^2=0\}$;
  \item $u-Q\geq 0$ in $\C^n$; and $\R^n=\{z\in \C^n: u(z)-Q(z)=0\}$; 
  \item for each $x,y\in \R^n$
  $$\lim_{t\to
0^+}\frac{Q(x+ity)-Q(x)}{t} =0.$$
\end{enumerate}
Item 1. is clear; 2. may be verified by direct calculation (the inequality also follows from the observation that $Q\in L(\C^n)$ and $Q$ equals $u$ on $\R^n$); and for 3., observe that
$$|1+(x+ity)^2|^2 = (1+x^2-t^2y^2)^2+4t^2(x\cdot y)^2=(1+x^2)^2+0(t^2)$$
so that 
$$Q(x+ity)-Q(x)=\frac{1}{2}\log |1+(x+ity)^2|-\frac{1}{2}\log (1+x^2)$$
$$=\frac{1}{4} \log \frac{(1+x^2)^2+0(t^2)}{(1+x^2)^2}\approx \frac{1}{4}\frac{0(t^2)}{(1+x^2)^2} \ \hbox{as} \ t\to 0. $$
Thus 1. and 2. imply that $v:=u-Q$ defines a nonnegative plurisubharmonic function in $\C^n\setminus V$, in particular, on a neighborhood $D\subset \C^n$ of $\R^n$; from 1., 
\begin{equation} \label{maeq} (dd^cv)^n = (dd^cu)^n \ \hbox{on} \ D;\end{equation}
and from 3., for each $x,y\in \R^n$
  $$\lim_{t\to
0^+}\frac{v(x+ity)-v(x)}{t} =\lim_{t\to
0^+}\frac{u(x+ity)-Q(x+ity)-u(x)+Q(x)}{t}$$
$$=\lim_{t\to 0^+}\frac{u(x+ity)-u(x)}{t}- \lim_{t\to
0^+}\frac{Q(x+ity)-Q(x)}{t}=\delta_u(x;y).$$
Then (\ref{maeq}), (\ref{dense}) and Proposition \ref{propbaran} give (\ref{monge}).
\end{proof}

{\bf Authors:}\\[\baselineskip]
L. Bos, leonardpeter.bos@univr.it\\
University of Verona, Verona, ITALY
\\[\baselineskip]
N. Levenberg, nlevenbe@indiana.edu\\
Indiana University, Bloomington, IN, USA\\
\\[\baselineskip]
S. Ma`u, s.mau@auckland.ac.nz\\
University of Auckland, Auckland, NEW ZEALAND\\
\\[\baselineskip]
F. Piazzon, fpiazzon@math.inipd.it\\
University of Padua, Padua, ITALY
\\[\baselineskip]


\bigskip

\begin{thebibliography}{GGK2}

\bibitem {bar} M. Baran, Plurisubharmonic extremal functions
and complex foliations for the complement of convex sets in
$\R^n$, \emph{Michigan Math. J.} {\textbf {39}} (1992), 395-404.


\bibitem {bt} E. Bedford, B. A. Taylor, The complex equilibrium
measure of a symmetric convex set in $\R^n$, \emph{Trans. AMS}
{\textbf{294}}  (1986), 705-717.

\bibitem{Bloomtams}  T. Bloom, Weighted polynomials and weighted pluripotential theory, \emph{Trans. Amer. Math. Society}, \textbf{361} (2009), 2163-2179.

\bibitem{BL} T. Bloom and N. Levenberg, Weighted pluripotential theory in $\C^n$, \emph{American J. Math.}  \textbf{125} (2003), no. 1, 57-103.

\bibitem {blw} L. Bos, N. Levenberg and S. Waldron, Pseudometrics,
distances, and multivariate polynomial inequalities, \emph{Journal of Approximation Theory} {\textbf{153}} (2008), no. 1, 80-96.

\bibitem{bs} M. Branker and M. Stawiska, Weighted pluripotential theory on compact K\"ahler manifolds, \emph{Ann. Polon. Math.} \textbf{95} (2009), no. 2, 163-177.

\bibitem{blm} D. Burns, N. Levenberg and S. Ma'u, Pluripotential
theory for convex bodies in $\R^N$,  \emph{Math.
Zeitschrift} {\textbf{250}}  (2005),  no. 1, 91-111.

\bibitem{blm2} D. Burns, N. Levenberg and S. Ma'u, Exterior Monge-Amp\`ere solutions, \emph{Advances in Math.} \textbf{222} (2009), Issue 2, 331-358.

\bibitem{blmr} D. Burns, N. Levenberg, S. Ma'u and S. Revesz, Monge-Amp\`ere measures for convex bodies and Bernstein-Markov type inequalities, \emph{Trans. Amer. Math. Soc.} \textbf{362} (2010), 6325-6340.

\bibitem{DS} T. C. Dinh and N. Sibony, 
Distribution des valeurs de transformations m\'eromorphes et applications, \emph{Comment. Math. Helv.}, \textbf{81} (2006), no. 1, 221-258.


\bibitem{GZ} V. Guedj and A. Zeriahi, Intrinsic capacities on compact K\"ahler manifolds, \emph{J. Geom. Anal.}, \textbf {15} (2005), no. 4, 607-639.

\bibitem{Kl} M. Klimek, {\sl Pluripotential Theory}, Clarendon
Press, Oxford, 1991.

\bibitem{lunpre} M. Lundin, The  extremal plurisubharmonic
function for the complement of the disk in $\R^2$, unpublished
preprint, 1984.


\bibitem{Sad}  A. Sadullaev, An estimate for polynomials
on analytic sets, \emph{Math.\ USSR-Izv.}, \textbf{20} (1983), 493-502.

\bibitem{Z} S. Zelditch, Pluripotential theory on Grauert tubes of real analytic Riemannian manifolds, I, in {\it Spectral geometry, 
Proc. Sympos. Pure Math.}, \textbf{84}, Amer. Math. Soc., 299-339.


\end{thebibliography}
\end{document}